\documentclass[11pt,twoside]{amsart}
\usepackage{latexsym,amssymb,amsmath}

\numberwithin{equation}{section}
\newtheorem{theorem}{Theorem}[section]
\newtheorem{lemma}[theorem]{Lemma}
\newtheorem{proposition}[theorem]{Proposition}
\newtheorem{corollary}[theorem]{Corollary}
\newtheorem{conjecture}[theorem]{Conjecture}

\theoremstyle{definition}
\newtheorem{definition}[theorem]{Definition} 
 
\newtheorem{remark}[theorem]{Remark}
\newtheorem{example}[theorem]{Example}

\begin{document}

%%%%%%%%%%%%%%%%%%%%%%%%%%%%%%%%%%%%%%%%%%%%%%%%%%%%%%%%%%%%%%%%%%%%%

\title[Algebraic and
combinatorial properties of ideals and algebras]{Algebraic and
combinatorial properties of ideals and 
algebras of uniform clutters of TDI systems} 

\author{Luis A. Dupont}
\address{
Departamento de
Matem\'aticas\\
Centro de Investigaci\'on y de Estudios Avanzados del
IPN\\
Apartado Postal
14--740 \\
07000 Mexico City, D.F. }\email{ldupont@math.cinvestav.mx}
\thanks{The second author was partially supported by CONACyT 
grant 49251-F and SNI}

\author{Rafael H. Villarreal}
\address{
Departamento de
Matem\'aticas\\
Centro de Investigaci\'on y de Estudios
Avanzados del
IPN\\
Apartado Postal
14--740 \\
07000 Mexico City, D.F.
}
\email{vila@math.cinvestav.mx}
\urladdr{http://www.math.cinvestav.mx/$\sim$vila/}


\begin{abstract} Let $\mathcal{C}$ be a uniform clutter 
and let $A$ be the incidence
matrix of 
$\mathcal{C}$. We denote the column vectors of $A$ by
$v_1,\ldots,v_q$. Under certain conditions we 
prove that $\mathcal{C}$ is vertex critical. If $\mathcal C$ satisfies
the max-flow 
min-cut property, we 
prove that $A$ diagonalizes over $\mathbb{Z}$ to an identity matrix
and that  
$v_1,\ldots,v_q$ form a Hilbert basis.  We also prove that if
$\mathcal C$ 
has a perfect matching such that $\mathcal{C}$ has the packing
property and its vertex covering number is equal to $2$, then $A$
diagonalizes over $\mathbb{Z}$ to an identity matrix. If $A$ is a
balanced matrix we 
prove that any regular triangulation of the cone generated by
$v_1,\ldots,v_q$ is unimodular. Some examples are presented to 
show that our results only hold for uniform clutters. 
These results are closely related to certain algebraic 
properties, such as the normality or torsion-freeness, of blowup
algebras of edge ideals and to finitely generated abelian groups. 
They are also related to the theory of Gr\"obner bases of toric
ideals and to Ehrhart rings.
\end{abstract}

\maketitle

\section{Introduction}
Let $R=K[x_1,\ldots,x_n]$ be a polynomial ring 
over a field $K$ and let $I$ be an ideal 
of $R$ of height $g$ minimally generated by a finite set 
$$
F=\{x^{v_1},\ldots,x^{v_q}\}
$$
of square-free monomials. As usual we
use the notation $x^a:=x_1^{a_1} \cdots x_n^{a_n}$, 
where $a=(a_i)\in \mathbb{N}^n$. The {\it support\/} of
$x^a$ is given by ${\rm supp}(x^a)= \{x_i\, |\, a_i>0\}$. For
technical reasons we 
shall assume that each variable $x_i$ occurs in at least one
monomial of $F$. 

A {\it clutter\/} $\mathcal{C}$ with finite vertex set 
$X$ is a family $E$ of subsets of $X$, called edges, none 
of which is included in another. The set of
vertices and edges of $\mathcal{C}$ are denoted by $X=V(\mathcal{C})$
and $E=E(\mathcal{C})$ respectively. Clutters are special types of
hypergraphs and are sometimes called 
{\em Sperner families} in the literature. One example of a clutter is
a graph with the vertices 
and edges defined in the usual way for graphs.
For a thorough study of clutters and
hypergraphs from the point of view of combinatorial optimization see
\cite{cornu-book,Schr2}.   

We associate to the 
ideal $I$ a {\it clutter\/} $\mathcal C$ by taking the set 
of indeterminates $X=\{x_1,\ldots,x_n\}$ as vertex set and 
$E=\{f_1,\ldots,f_q\}$ as edge set, where $f_k$ is the support of
$x^{v_k}$. The assignment $I\mapsto \mathcal C$ gives a natural one to one
correspondence between the family of square-free monomial ideals and
the family of clutters. The ideal $I$ is called the {\it edge
ideal\/} of $\mathcal{C}$. 
To stress the relationship 
between $I$ and $\mathcal C$ we will use the notation $I=I({\mathcal C})$. 
The $\{0,1\}$-vector $v_k$ is the so called {\it characteristic
vector\/} or {\it incidence vector\/} of $f_k$, i.e., 
$v_k=\sum_{x_i\in f_k}e_i$, where $e_i$ is the $i${\it th} unit
vector. We shall always assume that $\mathcal{C}$ has no isolated
vertices, i.e., each vertex $x_i$ occurs in at least one edge of
$\mathcal{C}$.  

Let $A$ be the {\it incidence matrix\/} of $\mathcal C$ whose column 
vectors are $v_1,\ldots,v_q$. The {\it set covering polyhedron\/} 
 of ${\mathcal C}$ is given by:
$$
Q(A)=\{x\in\mathbb{R}^n\vert\, x\geq 0;\, xA\geq{\mathbf 1}\},
$$
where $\mathbf{1}=(1,\ldots,1)$. A subset $C\subset X$ is called a 
{\it minimal vertex cover\/} of $\mathcal C$ if: 
(i) every edge of $\mathcal C$ contains at least one vertex of $C$, 
and (ii) there is no proper subset of $C$ with the first 
property. The map $C\mapsto \sum_{x_i\in C}e_i$ gives a bijection
between the minimal vertex covers of $\mathcal{C}$ and the integral
vertices of $Q(A)$. A polyhedron is called an {\it
integral polyhedron\/} if it has only integral vertices. A clutter is called 
$d$-{\it uniform\/} or {\it uniform\/} if all its edges have 
exactly $d$ vertices.

The contents of this paper are as follows. We begin in
Section~\ref{section2} by introducing various combinatorial 
properties of clutters. We then give a simple combinatorial proof
of the following result of \cite{reesclu}: 

\medskip
\noindent {\bf Lemma~\ref{nov19-03}, Proposition~\ref{nov20-03}.}{\em\ 
If $\mathcal{C}$ is a 
$d$-uniform clutter whose set covering polyhedron $Q(A)$ is integral, then 
there are $X_1,\ldots,X_d$ mutually disjoint minimal vertex covers of
$\mathcal C$ such  
that $X=\cup_{i=1}^d X_i$. In particular $|{\rm supp}(x^{v_i})\cap
X_k|=1$ for 
all $i,k$.}

\medskip

The original proof of this result was algebraic, it was based on the
fact that the radical of the ideal $IR[It]$ can be expressed in terms
of the minimal primes of $I$, where $R[It]$ is the Rees 
algebra of $I$ (see Section~\ref{algebras-tdi}).
Example~\ref{double-contraej} shows that this result 
fails if we drop the uniformity hypothesis. 
For use below we denote the smallest number of vertices in any 
minimal vertex cover of $\mathcal C$ by $\alpha_0({\mathcal C})$. The
clutter obtained from $\mathcal{C}$ by
deleting a vertex $x_i$ and removing all edges containing $x_i$ is
denoted by $\mathcal{C}\setminus\{x_i\}$. A clutter $\mathcal C$ is 
called {\it vertex critical\/} if $\alpha_0({\mathcal
C}\setminus\{x_i\})<\alpha_0({\mathcal C})$ 
for all $i$. A set of pairwise disjoint
edges of $\mathcal C$ is called 
{\it independent} or a {\it matching\/} and a set
of independent edges of $\mathcal{C}$ whose union is $X$ is called a 
{\it perfect matching\/}.
We then prove:

\medskip

\noindent {\bf Proposition~\ref{integral+uniforme+pm}.}{\it\ 
Let $\mathcal{C}$ be a $d$-uniform clutter with a perfect matching such
that $Q(A)$ is integral. Then $\mathcal C$ is vertex critical.}

\medskip

A simple example is shown to see that this result fails for non uniform
clutters with integral set covering polyhedron 
(Remark~\ref{double-contraej-re}). 

In Section~\ref{algebras-tdi} we
introduce Rees algebras, Ehrhart rings, and edge subrings. 
Certain algebraic properties of these graded algebras
such as the normality and torsion-freeness are related to
combinatorial optimization  properties of clutters such as the
max-flow min-cut property (see Definition~\ref{mfmc-def}) and the
integrality of $Q(A)$ 
\cite{reesclu,clutters}. This relation between algebra and 
combinatorics will be quite useful here. In Theorems~\ref{intiffint},
\ref{noclu1},  
 and Proposition~\ref{two-char} we summarize the 
algebro-combinatorial facts needed to show the main result of
Section~\ref{algebras-tdi}:  

\medskip
\noindent {\bf Theorem~\ref{mfmch->diag}.}{\it\ If $\mathcal C$ is a
uniform clutter 
with the max-flow min-cut property, then 

{\rm (a)} $\Delta_r(A)=1$, where $r={\rm rank}(A)$.

{\rm (b)} $\mathbb{N}{\mathcal A}=\mathbb{R}_+{\mathcal
A}\cap\mathbb{Z}^n$, where ${\mathcal A}=\{v_1,\ldots,v_q\}$.} 
\medskip

Here $\Delta_r(A)$ denotes the greatest common divisor of all the
nonzero $r\times r$ sub-determinants of $A$, 
$\mathbb{N}{\mathcal A}$ denotes the semigroup generated by $\mathcal{A}$,
and $\mathbb{R}_+{\mathcal A}$ denotes the cone generated by
$\mathcal{A}$. Condition (b) means that $\mathcal{A}$ is a {\it Hilbert
basis\/} for $\mathbb{R}_+\mathcal{A}$. As interesting consequences 
we obtain that if 
$\mathcal{C}$ is a $d$-uniform clutter
with the max-flow min-cut property, then $A$ diagonalizes over
$\mathbb{Z}$---using row and column operations---to an identity matrix (see
Corollary~\ref{mfmc->smith=identity}) and $\mathcal{C}$ has a perfect
matching if and only if $n=d\alpha_0(\mathcal{C})$ (see
Corollary~\ref{pm-iff-n=gd}). In Example~\ref{unif-ess-mfmch} we show that 
the uniformity hypothesis is essential in the two statements of 
Theorem~\ref{mfmch->diag}. 

Section~\ref{section-packing-partition} deals with the 
diagonalization problem (see Conjecture~\ref{git-val-vi}) for
clutters with the packing 
property (see Definition~\ref{pp-def}). The following is one of the
main results of this section, it gives some support to
Conjecture~\ref{git-val-vi}.

\medskip

\noindent {\bf Theorem~\ref{jan18-04}.}{\it\ Let $\mathcal{C}$ be a 
$d$-uniform clutter with a perfect matching such that
$\mathcal{C}$ has the packing property and $\alpha_0(\mathcal{C})=2$.
If $A$ has rank $r$, 
then 
$$
\Delta_r\left(\hspace{-2mm}
\begin{array}{c}A\\ 
\mathbf{1}\end{array}\hspace{-2mm}\right)=1.
$$}

As an application we obtain the next result which gives some support
to a Conjecture of Conforti and Cornu\'ejols \cite{CC} (see
Conjecture~\ref{conforti-cornuejols1}).  

\medskip

\noindent {\bf Corollary~\ref{cc-conj-supp}.}{\it\ Let $\mathcal{C}$ be a 
$d$-uniform clutter with a perfect matching such that 
$\mathcal{C}$ has the packing property and $\alpha_0(\mathcal{C})=2$.
If $v_1,\ldots,v_q$ are linearly independent, then $\mathcal C$ has the
max-flow min-cut property.}

\medskip

All clutters of Section~\ref{section-packing-partition} satisfy
the hypotheses of Theorem~\ref{another-mt}. We call this type of
clutters $2$-{\it partitionable} (see Example~\ref{q6}). They occur
naturally in the theory of blockers of unmixed bipartite graphs (see
Corollary~\ref{may17-09}). The
other main result of 
Section~\ref{section-packing-partition} is about some of the
properties of this family of clutters: 

\medskip
\noindent {\bf Theorem~\ref{another-mt}.}{\it\ Let $\mathcal{C}$ be a
$d$-uniform clutter with a 
partition $X_1,\ldots,X_d$ of $X$ such that $X_i$ is a minimal vertex
cover of $\mathcal{C}$ and $|X_i|=2$ for all $i$. Then {\rm (a)}
${\rm rank}(A)\leq d+1$.  
{\rm (b)} If $C$ is a minimal vertex cover of $\mathcal{C}$, 
then $2\leq|C|\leq d$. {\rm(c)} If $\mathcal C$ satisfies the K\"onig
property and there is 
a minimal vertex cover $C$ of $\mathcal C$ with $|C|=d\geq 3$, then
${\rm rank}(A)=d+1$. {\rm (d)} If $I=I(\mathcal{C})$ is minimally
non-normal and 
$\mathcal{C}$ satisfies the packing property, then ${\rm rank}(A)=d+1$.
}

\medskip

Regular and unimodular triangulations are introduced in 
Section~\ref{section-trian-balanced}. Let $K[F]$ be the monomial
subring of 
$R$ generated by $F=\{x^{v_1},\ldots,x^{v_q}\}$. There is a 
relationship between the Gr\"obner bases of the toric ideal of 
$K[F]$ and the triangulations of ${\mathcal A}=\{v_1,\ldots,v_q\}$, 
which has many interesting
applications. We make use of the theory of Gr\"obner bases and convex
polytopes, which was created and developed by Sturmfels \cite{Stur1},
to prove the following main result of
Section~\ref{section-trian-balanced}:

\medskip

\noindent {\bf Theorem~\ref{balanced->urt}.}{\it\ Let $A$ be a 
balanced matrix with distinct  
column vectors $v_1,\ldots,v_q$. If $|v_i|=d$ for all $i$, then any  
regular triangulation of the cone $\mathbb{R}_+\{v_1,\ldots,v_q\}$ 
is unimodular.}

\medskip

Here $|v_i|$ denotes the sum of the entries of the vector $v_i$. 
Recall that a matrix $A$ with entries in $\{0,1\}$ is 
called {\it balanced\/} if $A$ has no square submatrix of odd order
with exactly two $1$'s in each row and column. If we do not require
the uniformity condition $|v_i|=d$ for all $i$ this result is false,
as is seen in Example~\ref{balanced-nu-example}. What makes this
result surprising is the fact that not all balanced matrices are  
unimodular (see Example~\ref{balanced-nu-example}). 
This result gives some support to
Conjecture~\ref{vila-mfmc-unimodular}: If $\mathcal C$ is a 
uniform clutter that satisfies the max-flow 
min-cut property, then the rational polyhedral cone 
$\mathbb{R}_+\{v_1,\ldots,v_q\}$ has a unimodular regular 
triangulation. 

Throughout the paper we introduce most of the 
notions that are relevant for our purposes. For unexplained
terminology and 
notation we refer to  
\cite{Schr2} (for the theory of combinatorial optimization) 
and \cite{BHer,bookthree} (for the theory of blowup algebras and
integral closures). See \cite{Mat}
for additional information about commutative rings and ideals.  

\section{On the structure of ideal uniform clutters}\label{section2}

We continue to use the notation and definitions used in the
introduction. In what follows $\mathcal{C}$ denotes a $d$-uniform
clutter with vertex set
$X=\{x_1,\ldots,x_n\}$, edge set $E(\mathcal{C})$, edge ideal
$I=I(\mathcal{C})$, and incidence matrix
$A$. The column vectors of $A$ are denoted by $v_1,\ldots,v_q$ and
the edge  ideal of $\mathcal{C}$ is given by
$I=(x^{v_1},\ldots,x^{v_q})$.  

In this section we study the structure of uniform clutters whose set
covering polyhedron is integral. Examples of this type of clutters 
include bipartite graphs and uniform clutters with the max-flow
min-cut property. A clutter whose set covering polyhedron is 
integral is called {\it ideal\/} in the literature \cite{cornu-book}.
We denote the smallest number of vertices in any 
minimal vertex cover of $\mathcal C$ by $\alpha_0({\mathcal C})$ and the 
maximum number of independent edges of ${\mathcal C}$ by 
$\beta_1({\mathcal C})$. These two numbers are called the 
{\it vertex covering number\/} and the {\it edge independence
number\/} respectively. Notice that in general 
$\beta_1({\mathcal C})\leq \alpha_0({\mathcal C})$. If equality
occurs we say 
that $\mathcal{C}$ has the {\it K\"onig property}. 

Recall that $\mathfrak{p}$ is a minimal prime of $I =I(\mathcal{C})$
if and only if  
$\mathfrak{p}=(C)$ for some minimal vertex cover $C$ of $\mathcal{C}$ 
\cite[Proposition~6.1.16]{monalg}, where $(C)$ is the ideal of $R$
generated by $C$. 
Thus the primary decomposition of
the edge ideal of $\mathcal{C}$ is given by
$$
I(\mathcal{C})=(C_1)\cap (C_2)\cap\cdots\cap (C_s),
$$
where $C_1,\ldots,C_s$ are the minimal vertex 
covers of $\mathcal{C}$. 
In particular observe that 
${\rm ht}\, I(\mathcal{C})$, the height of 
$I(\mathcal{C})$, equals the number of vertices in a minimum vertex
cover of $\mathcal{C}$, i.e., ${\rm ht}\,
I(\mathcal{C})=\alpha_0(\mathcal{C})$. This is a hint of the rich 
interaction between the combinatorics of
$\mathcal{C}$ 
and the algebra of $I(\mathcal{C})$.

\medskip

The next result was shown in \cite{reesclu} using commutative algebra
methods.  
Here we give a simple combinatorial proof. 

\begin{lemma}{\rm\cite{reesclu}}\label{nov19-03} If ${\mathcal C}$ is a
$d$-uniform clutter such that $Q(A)$ is integral, then there exists a
minimal vertex cover of $\mathcal C$  
intersecting every edge of $\mathcal C$ in exactly one vertex. 
\end{lemma}

\begin{proof}  Let $B$ be the integral matrix whose columns are the
vertices of $Q(A)$. It is not hard to show that {\it a vector 
$\alpha\in\mathbb{R}^n$ is an integral vertex of $Q(A)$ if and only if 
$\alpha=\sum_{x_i\in C}e_i$ for some minimal vertex cover 
$C$ of $\mathcal{C}$}. Thus the columns of $B$ are the characteristic
vectors of the minimal vertex covers of $\mathcal C$. 
Using \cite[Theorem~1.17]{cornu-book} we get that 
$$
Q(B)=\{x\vert\,
x\geq 0;xB\geq\mathbf{1}\}
$$ 
is an integral
polyhedron whose vertices are the columns of $A$. Therefore we have
the equality 
\begin{equation}\label{may10-09}
Q(B)=\mathbb{R}_+^n+{\rm conv}(v_1,\ldots,v_q).
\end{equation}

We proceed by contradiction. Assume that for each column $u_k$
of $B$ there 
exists a vector $v_{i_k}$ in $\{v_1,\ldots,v_q\}$ such that
$\langle v_{i_k},u_k\rangle\geq 2$. Then
$$
v_{i_k}B\geq \mathbf{1}+e_k.
$$
Consider the vector $\alpha=v_{i_1}+\cdots+v_{i_s}$, where $s$ is the
number of columns of $B$. From the
inequality 
$$
{\alpha}B\geq (\mathbf{1}+e_1)+\cdots+(\mathbf{1}+e_s)=(s+1,\ldots,s+1)
$$
we obtain that $\alpha/(s+1)\in Q(B)$. Thus, using
Eq.~(\ref{may10-09}), 
we can write 
\begin{equation}\label{may19-09}
\alpha/(s+1)=\mu_1e_1+\cdots+\mu_ne_n+\lambda_1v_1+\cdots+\lambda_qv_q\
\ \ \ (\mu_i,\lambda_j\geq 0;\ \textstyle\sum\lambda_i=1).
\end{equation}
Therefore taking inner products with $\mathbf{1}$ in
Eq.~(\ref{may19-09}) and using the fact that $\mathcal{C}$ is $d$-uniform 
we get that $|\alpha|\geq (s+1)d$. Then using the equality
$\alpha=v_{i_1}+\cdots+v_{i_s}$ we conclude
$$
sd=|v_{i_1}|+\cdots+|v_{i_s}|=|\alpha|\geq (s+1)d,
$$
a contradiction because $d\geq 1$. \end{proof}  

A graph $G$ is called {\em strongly perfect} if every induced 
subgraph $H$ of $G$ has a maximal independent set of vertices $F$ such that 
$|F\cap K|=1$ for any maximal clique $K$ of $H$. Bipartite and 
chordal graphs are strongly perfect. If $A$ is the vertex-clique
matrix of a graph $G$, then $G$  
being strongly perfect implies that the clique polytope of $G$,  
$\{x\vert\, x \geq 0;\,  xA\leq\mathbf{1} \}$, has a vertex 
that intersects every maximal clique.  In this sense, uniform
clutters such that $Q(A)$ is integral 
can be thought of as being analogous to strongly perfect graphs.

The notion of a minor plays a prominent role in 
combinatorial optimization \cite{cornu-book}. Recall that a proper
ideal $I'$ of $R$ is called a {\it minor\/} of $I=I(\mathcal{C})$ if there 
is a subset 
$$
X'=\{x_{i_1},\ldots,x_{i_r},x_{j_1},\ldots,x_{j_s}\}
$$
of the set of variables $X$ such that $I'$ is obtained from $I$ by 
making $x_{i_k}=0$ and $x_{j_\ell}=1$ for all $k,\ell$. Notice that a set
of generators $x^{v_1'},\ldots,x^{v_q'}$ of $I'$ is 
obtained from a set of generators $x^{v_1},\ldots,x^{v_q}$ of
$I$ by making $x_{i_k}=0$ and $x_{j_\ell}=1$ for all $k,\ell$. The
ideal $I$ is 
considered itself a minor.  A clutter $\mathcal{C}'$ is called a {\it
minor\/} of $\mathcal C$ if $\mathcal{C}'$ 
corresponds to a minor $I'$ of $I$ under the correspondence between
square-free monomial ideals and clutters.  This terminology is 
consistent with that of \cite[p.~23]{cornu-book}. Also notice 
that ${\mathcal C}'$ is obtained from $I'$ by considering the unique set 
of square-free monomials of $R$ that minimally generate $I'$. The
clutter ${\mathcal C}\setminus\{x_i\}$ corresponds to the ideal $I'$
obtained from $I$ 
by making $x_i=0$, i.e., $\mathcal{C}\setminus\{x_i\}$ is a special
type of a minor which is called a {\it deletion\/}.

The notion of a minor of a clutter is not a generalization of the notion of a
minor of a graph in the sense of graph theory \cite[p.~25]{Schr2}.
For instance if  
$G$ is a cycle of length four and we contract an edge we obtain that 
a triangle is a minor of $G$, but a triangle cannot be a minor of $G$
in our sense.  

\begin{proposition}\label{nov20-03} If $ \mathcal{C}$ is a
$d$-uniform clutter whose set covering polyhedron $Q(A)$ is integral, then 
there are $X_1,\ldots,X_d$ mutually disjoint minimal vertex covers of
$\mathcal C$ such  
that $X=\cup_{i=1}^d X_i$. 
\end{proposition}

\begin{proof}  By induction on $d$. If $d=1$, then
$E(\mathcal{C})=\{\{x_1\},\ldots,\{x_n\}\}$ and $X$ is a minimal
vertex cover of $\mathcal{C}$. In this case we set $X_1=X$. 
Assume $d\geq 2$. By
Lemma~\ref{nov19-03} there is a minimal vertex  
cover $X_1$ of $\mathcal C$ such that $|{\rm supp}(x^{v_i})\cap
X_1|=1$ for all  
$i$. Consider the ideal $I'$ obtained from $I$ by making $x_i=1$ for 
$x_i\in X_1$. Let $\mathcal{C}'$ be the clutter corresponding to $I'$
and let $A'$ be the incidence matrix of $\mathcal{C}'$. The ideal $I'$
(resp. the clutter $\mathcal{C}'$) is a minor of $I$ (resp.
$\mathcal{C}$). Recall that the integrality of $Q(A)$ is preserved
under taking 
minors \cite[Theorem~78.2]{Schr2}, so $Q(A')$ is integral. Then
$\mathcal{C}'$ 
is a $(d-1)$-uniform clutter whose set covering polyhedron $Q(A')$ is
integral. Note that $V(\mathcal{C}')=X\setminus X_1$. 
Therefore by induction hypothesis there are $X_2,\ldots,X_d$
pairwise disjoint minimal vertex covers of $\mathcal{C}'$ such that 
$X\setminus X_1=X_2\cup\cdots\cup X_d$. To complete the proof 
observe that $X_2,\ldots,X_d$ are minimal vertex covers of
$\mathcal{C}$. Indeed if $e$ is an edge of $\mathcal{C}$ and $2\leq
k\leq d$, then $e\cap X_1=\{x_i\}$ for some $i$. Since
$e\setminus\{x_i\}$ is an edge of $\mathcal{C}'$, we get
$(e\setminus\{x_i\})\cap X_k\neq\emptyset$. Hence $X_k$ is a vertex cover of
$\mathcal{C}$. Furthermore if $x\in X_k$, then by the minimality of
$X_k$ relative to $\mathcal{C}'$ there is an edge $e'$ of
$\mathcal{C}'$ disjoint from 
$X_k\setminus\{x\}$. 
Since $e=e'\cup \{y\}$ is an edge of $\mathcal{C}$ for some $y\in
X_1$, we obtain that $e$ is an edge of $\mathcal{C}$ disjoint from 
$X_k\setminus\{x\}$. Therefore $X_k$ is a minimal vertex cover of
$\mathcal{C}$, as required.  \end{proof}  

\begin{example}\label{double-contraej}\rm 
Consider the clutter $\mathcal{C}$ with vertex set 
$X=\{x_1,\ldots,x_9\}$ whose edges are 
$$
\begin{array}{lllll}
f_1=\{x_1,x_2\},&f_2=\{x_3,x_4,x_5,x_6\},&f_3=\{x_7,x_8,x_9\},&  &
\\ 
f_4=\{x_1,x_3\},&f_5=\{x_2,x_4\},&f_6=\{x_5,x_7\},&f_7=\{x_6,x_8\}.&
\end{array}
$$
In this example $Q(A)$ is integral because the incidence matrix of
$\mathcal{C}$ is a balanced matrix. However 
$|C\cap f_i|\geq 2$ for any minimal vertex cover $C$ and for any $i$.
Thus the uniformity hypothesis is essential in
Proposition~\ref{nov20-03}. 
\end{example}

\begin{definition}\rm Let $X_1,\ldots,X_d$ be a partition of $X$. 
The matroid  whose collection of bases is 
$$
{\mathcal B}=\{\{y_1,\ldots,y_d\}\vert\, y_i\in X_i\, \mbox{ for }i=1,\ldots,d\}
$$ 
is called the {\it transversal matroid\/} defined by $X_1,\ldots,X_d$
and is denoted by $\mathcal{M}$.
\end{definition}

Recall that the set covering polyhedron of the clutter of bases of any 
transversal matroid is integral \cite[p.~92]{reesclu}. The next result 
generalizes the fact that
any bipartite graph is a subgraph of a complete bipartite graph. 

\begin{corollary} If $\mathcal{C}$ is a $d$-uniform clutter and $Q(A)$ is
integral, then there is a partition $X_1,\ldots,X_d$ of $X$ such that
$\mathcal C$ is a 
subclutter of the clutter of bases of the transversal
matroid $\mathcal{M}$ defined by $X_1,\ldots,X_d$. 
\end{corollary}

The notion of a vertex critical clutter is the natural generalization
of the corresponding notion for graph \cite{Har}. The family of
vertex critical graphs has many nice properties, for instance if $G$
is a vertex critical graph with $n$ vertices, then $\alpha_0(G)\geq
n/2$ (see \cite[Theorem~3.11]{bounds}). If a clutter
$\mathcal{C}$ is vertex critical, then 
$\alpha_0({\mathcal C})=\alpha_0({\mathcal C}\setminus\{x_i\})+1$ for
all $i$. 

\begin{proposition}\label{integral+uniforme+pm} 
Let $\mathcal{C}$ be a $d$-uniform clutter with a perfect matching such
that $Q(A)$ is integral. Then $\mathcal C$ is vertex critical.
\end{proposition}

\begin{proof}  First, we claim that $n=gd$, where $g={\rm ht}\,
I(\mathcal{C})$. First 
we show that $n\geq gd$. Notice that $\mathbf{1}A\geq d\mathbf{1}$, i.e.,
$\mathbf{1}/d\in Q(A)$. Let $u_1,\ldots,u_s$ be the characteristic
vectors of the minimal vertex covers of $\mathcal C$. As $Q(A)$ is
integral, the vertices of $Q(A)$ are the $u_i$'s. Hence we have the
equality 
$$
Q(A)=\mathbb{R}_+^n+{\rm
conv}(u_1,\ldots,u_s).
$$ 
Since $\mathbf{1}/d\in Q(A)$, using this equality we get 
$$
\mathbf{1}/d=\delta+\lambda_1u_1+\cdots+\lambda_su_s;\ \ \ 
(\delta\in\mathbb{R}_+^n;\, \lambda_i\geq 0;\, 
\textstyle\sum_i\lambda_i=1).
$$
Therefore $n\geq gd$. By hypothesis there are mutually disjoint edges
$f_1,\ldots,f_r$ such that $X$ is equal to $f_1\cup\cdots\cup f_r$.
Consequently  $n=rd$. So $n=rd\geq gd$, i.e., $r\geq g$. On the other hand 
$g={\rm ht}\, I(\mathcal{C})\geq \beta_1(\mathcal{C})\geq r$. Thus $r=g$
and $n=gd$ as claimed. In particular $\mathcal{C}$ has the K\"onig
property. 
We now prove that $\mathcal{C}$ is vertex critical.
By Proposition~\ref{nov20-03} there are $X_1,\ldots,X_d$ mutually
disjoint minimal vertex covers of $\mathcal C$ such 
that $X=\cup_{i=1}^d X_i$. Hence 
$$
n=gd=|X_1|+\cdots+|X_d|.
$$
As $|X_i|\geq g$ for all $i$, we get $|X_i|=g$ for all $i$. It follows
rapidly that $\mathcal{C}$ is vertex critical. Indeed notice that each
vertex $x_i$ belongs to a minimal vertex cover $C_i$ of $\mathcal{C}$ with
$g$ vertices. The set $C_i\setminus\{x_i\}$ is a vertex cover of
$\mathcal{C}\setminus\{x_i\}$ of size $g-1$. Hence
$\alpha_0(\mathcal{C}\setminus\{x_i\})<\alpha_0(\mathcal{C})$.
\end{proof}

\begin{remark}\label{double-contraej-re}\rm 
Consider the clutter $\mathcal{C}$ of Example~\ref{double-contraej}. 
This clutter has a perfect matching and $Q(A)$ is integral, but it is not
vertex critical because $\alpha_0(\mathcal{C}\setminus\{x_9\})=
\alpha_0(\mathcal{C})=4$. Thus the uniformity condition is essential
in Proposition~\ref{integral+uniforme+pm}.
\end{remark}

From the proof of Proposition~\ref{integral+uniforme+pm} we get: 
\begin{proposition}\label{integral+uniforme+pm-cor} 
Let $\mathcal{C}$ be a $d$-uniform clutter with a perfect 
matching $f_1,\ldots,f_r$. 
If $Q(A)$ is integral, then $r=\alpha_0(\mathcal{C})$ and there are
$X_1,\ldots,X_d$ mutually
disjoint minimal vertex covers of $\mathcal C$ of 
size $\alpha_0(\mathcal{C})$ such that $X=\cup_{i=1}^d X_i$. 
\end{proposition}

\section{Algebras and TDI systems of uniform
clutters}\label{algebras-tdi}

As before let $R=K[x_1,\ldots,x_n]$ be a polynomial ring 
over a field $K$ and let $\mathcal{C}$ be a clutter with vertex set
$X=\{x_1,\ldots,x_n\}$, edge set $E(\mathcal{C})$, edge ideal
$I=I(\mathcal{C})$, and incidence matrix
$A$. The column vectors of $A$ are denoted by $v_1,\ldots,v_q$. Thus 
the edge ideal of $\mathcal{C}$ is the ideal of $R$ 
generated by the set $F=\{x^{v_1},\ldots,x^{v_q}\}$. 

First we examine the interaction between 
combinatorial optimization properties of clutters and algebraic
properties of monomial algebras. The {\it monomial algebras\/}
considered here are:  
(a) the {\it Rees algebra\/}
$$
R[It]:=R\oplus It\oplus\cdots\oplus I^{i}t^i\oplus\cdots
\subset R[t],
$$
where $t$ is a new variable, (b) the 
{\it homogeneous monomial subring\/} 
$$
K[Ft]=K[x^{v_1}t,\ldots,x^{v_q}t]\subset R[t]
$$
spanned by $Ft=\{x^{v_1}t,\ldots,x^{v_q}t\}$, (c) the 
{\it edge subring\/} 
$$
K[F]=K[x^{v_1},\ldots,x^{v_q}]\subset R
$$
spanned by $F$, and (d) the {\it Ehrhart ring\/} 
$$
A(P)=K[\{x^at^i\vert\, a\in \mathbb{Z}^n \cap iP; i\in
\mathbb{N}\}]\subset R[t] 
$$
of the lattice polytope $P={\rm conv}(v_1,\ldots,v_q)$. 

The Rees algebra of the edge ideal $I$ can be written as
\begin{eqnarray*}
R[It]&=&K[\{x^at^b\vert\, (a,b)\in\mathbb{N}{\mathcal A}'\}] 
\end{eqnarray*}
where ${\mathcal A}'=\{(v_1,1),\ldots,(v_q,1),e_1,\ldots,e_n\}$,
$e_i$ 
is the $i${\it th} unit vector in $\mathbb{R}^{n+1}$, and 
$\mathbb{N}{\mathcal A}'$ is the subsemigroup of 
$\mathbb{N}^{n+1}$ spanned by ${\mathcal A}'$. According to
\cite[Theorem 7.2.28]{monalg} the 
integral closure of $R[It]$ 
in its field of fractions can be expressed as
\begin{eqnarray*}
\overline{R[It]}&=&K[\{x^at^b\vert\, (a,b)\in
\mathbb{Z}{\mathcal
A}'\cap \mathbb{R}_+{\mathcal A}'\}]
\end{eqnarray*}
where $\mathbb{R}_+{\mathcal A}'$ is the cone 
spanned by ${\mathcal A}'$ and $\mathbb{Z}{\mathcal A}'$ is the subgroup
spanned by ${\mathcal A}'$. The cone $\mathbb{R}_+{\mathcal A}'$ is
called the 
{\it Rees cone\/} of $I$. The Rees algebra of $I$ is called 
{\it normal\/} if $R[It]=\overline{R[It]}$. Notice that
$\mathbb{Z}{\mathcal A}'=\mathbb{Z}^{n+1}$. Hence we obtain the
following well known fact:

\begin{lemma}\label{dec20-07} $R[It]$ is normal if and only if
$\mathbb{N}{\mathcal A}'= 
\mathbb{Z}^{n+1}\cap\mathbb{R}_+{\mathcal A}'$. 
\end{lemma}

The Rees cone of $I$ has dimension $n+1$ because $\mathbb{Z}{\mathcal
A}'=\mathbb{Z}^{n+1}$. According to \cite[Theorem~4.1.1]{webster}
there is a unique 
irreducible representation  
$$
{\mathbb R}_+{\mathcal A}'=H_{e_1}^+\cap H_{e_2}^+\cap\cdots\cap
H_{e_{n+1}}^+\cap H_{\ell_1}^+\cap H_{\ell_2}^+\cap\cdots\cap
H_{\ell_r}^+\nonumber
$$
such that each $\ell_k$ is in $\mathbb{Z}^{n+1}$, the non-zero
entries of 
each $\ell_k$ are relatively prime, and none of the closed 
halfspaces $H_{e_1}^+,\ldots,
H_{e_{n+1}}^+,H_{\ell_1}^+,\ldots,H_{\ell_r}^+$ can be
omitted from 
the intersection. Here $H_{a}^+$ denotes 
the closed halfspace
$$
H_a^+=\{x\in\mathbb{R}^{n+1}\vert\, \langle
x,a\rangle\geq 0\},
$$
$H_a$ stands for the hyperplane through the origin with normal
vector $a$, and $\langle\ ,\, \rangle$ denotes the standard 
inner product. Irreducible representations of Rees cones were
first introduced and studied in \cite{normali}. There are some 
interesting links between these representations, 
edge ideals \cite{reesclu}, perfect graphs \cite{perfect}, and 
bases monomial ideals of matroids or polymatroids \cite{matrof}.

\medskip

The Rees cone of $I$ and the set covering polyhedron of $\mathcal{C}$
are closely related:  

\begin{theorem}{\rm \cite[Corollary~3.13]{clutters}}\label{intiffint}
Let $C_1,\ldots,C_s$ be the minimal vertex 
covers of a clutter $\mathcal C$ and let $u_k=\sum_{x_i\in C_k}e_i$ for 
$1\leq k\leq s$. Then $Q(A)$ is
 integral if and only if the irreducible representation of the Rees
 cone is{\rm :}
\begin{equation}\label{okayama-car1} 
{\mathbb R}_+{\mathcal A}'=H_{e_1}^+\cap H_{e_2}^+\cap\cdots\cap
H_{e_{n+1}}^+\cap H_{\ell_1}^+\cap H_{\ell_2}^+\cap\cdots\cap
H_{\ell_s}^+,
\end{equation}
where $\ell_k=(u_k,-1)$ for
$1\leq k\leq s$.
\end{theorem}

\begin{definition}\label{mfmc-def}\rm The clutter $\mathcal C$
satisfies the {\it max-flow min-cut\/} 
(MFMC) 
property if both sides 
of the LP-duality equation
\begin{equation}\label{jun6-2-03-1}
{\rm min}\{\langle \alpha,x\rangle \vert\, x\geq 0; xA\geq{\mathbf 1}\}=
{\rm max}\{\langle y,{\mathbf 1}\rangle \vert\, y\geq 0; Ay\leq\alpha\} 
\end{equation}
have integral optimum solutions $x$ and $y$ for each non-negative 
integral vector $\alpha$. 
The system $xA\geq\mathbf{\mathbf 1}$; $x\geq 0$ is called 
{\it totally dual integral\/} (TDI) if the maximum has an integral
optimum solution $y$ for each integral vector $\alpha$ with 
finite maximum. 
\end{definition}

A breakthrough in the area of monomial algebras is the
translation of combinatorial problems (e.g., the Conforti-Cornu\'ejols 
conjecture \cite{cornu-book}, the max-flow min-cut property, or the
idealness of a clutter) into algebraic problems of monomial 
algebras \cite{reesclu,clutters}. A typical example is the 
following result that describes the max-flow min-cut property in
algebraic and optimization terms.  

\begin{theorem}{\rm\cite{normali,reesclu,clutters,HuSV,Schr2}}\label{noclu1}
The following statements are equivalent\/{\rm :}
\begin{enumerate}
\item[(i)] The associated graded ring ${\rm gr}_I(R)=R[It]/IR[It]$ is reduced.
\item[(ii)] $R[It]$ is normal and $Q(A)$ is an integral
polyhedron.
\item[(iii)]  $I^{i}=I^{(i)}$ for $i\geq 1$, where $I^{(i)}$ is
the $i${\it th} symbolic power of $I$.
\item[(iv)]$\mathcal C$ has the max-flow min-cut
property.
\item[(v)] $x \geq 0;\, xA \geq\mathbf{1}$ is a {\rm TDI system}.
\end{enumerate}
\end{theorem}

For an integral matrix $B\neq(0)$, the 
greatest common divisor of all the nonzero $r\times r$ sub-determinants 
of $B$ will be denoted by $\Delta_r(B)$.

\begin{proposition}\label{two-char} Let $\mathcal{C}$ be a 
clutter and let $B$ be the matrix with column vectors 
$(v_1,1),\ldots,(v_q,1)$. The following statements hold{\rm :}  
\begin{enumerate}
\item[(i) ] {\rm \cite[Proposition~4.4]{reesclu}} If
$\mathcal{C}$ is uniform, then $\mathcal{C}$
has the max-flow min-cut property if and only if $Q(A)$ 
is integral and $K[Ft]=A(P)$.
\item[(ii)] {\rm \cite[Theorem~3.9]{ehrhart}} 
$\Delta_r(B)=1$ if and only if $\overline{K[Ft]}=A(P)$, where $r$ is
the rank of $B$.  
\end{enumerate}
\end{proposition}

We come to the main result of this section.

\begin{theorem}\label{mfmch->diag} If $\mathcal C$ is a uniform clutter
with the max-flow min-cut property, then 
\begin{enumerate}
\item[(a)] $\Delta_r(A)=1$, where $r={\rm rank}(A)$.
\item[(b)] $\mathbb{N}{\mathcal A}=\mathbb{R}_+{\mathcal
A}\cap\mathbb{Z}^n$, where ${\mathcal A}=\{v_1,\ldots,v_q\}$. 
\end{enumerate}
\end{theorem}

\begin{proof}  (a) Let $\widetilde{A}$ be the matrix 
with column vectors $(v_1,0),\ldots,(v_q,0)$. We need only show 
that $\Delta_r(\widetilde{A})=1$ because $A$ and $\widetilde{A}$ have the same
rank and $\Delta_r(A)=\Delta_r(\widetilde{A})$. Let 
$B$ be the matrix  with column vectors $(v_1,1),\ldots,(v_q,1)$. 
Since the clutter $\mathcal{C}$ is uniform, the last row 
vector of $B$, i.e., the 
vector $\mathbf{1}=(1,\ldots,1)$, is a $\mathbb{Q}$-linear combination of the 
first $n$ rows of $B$. Thus $\widetilde{A}$ and $B$ have the same rank. 
 By Proposition~\ref{two-char}(i) we obtain 
$K[Ft]=A(P)$. In particular, taking integral closures, 
one has $\overline{K[Ft]}=A(P)$ because $A(P)$
is always a normal domain. Hence by Proposition~\ref{two-char}(ii) we
have $\Delta_r(B)=1$. Recall that $\Delta_r(\widetilde{A})=1$ if and only if 
$\widetilde{A}$ is equivalent over $\mathbb{Z}$ to an identity matrix. In
other words  $\Delta_r(\widetilde{A})=1$ if and only if all the
invariant factors of 
$\widetilde{A}$ are equal to $1$. Thus it suffices to prove that $B$
is equivalent  
to $\widetilde{A}$ over $\mathbb{Z}$.  Notice that in general 
$B$ and $\widetilde{A}$ are not equivalent over $\mathbb{Z}$ (for
instance if 
$\mathcal{C}$ is a cycle of length three, then $\widetilde{A}$ and
$B$ have rank 
$3$, $\Delta_3(\widetilde{A})=2$ and $\Delta_3(B)=1$). By
Proposition~\ref{nov20-03}, there are  
$X_1,\ldots,X_d$ mutually disjoint minimal vertex covers of $\mathcal
C$ such  
that $X=\cup_{i=1}^d X_i$ and
\begin{equation}
|{\rm supp}(x^{v_i})\cap X_k|=1\ \ \ \ \forall\ i,k.
\end{equation} 
By permuting the variables we may assume 
that $X_1$ is equal to $\{x_1,\ldots,x_r\}$. Hence the last row of
$B$, which is the vector $\mathbf{1}$, is the sum of the
first $|X_1|$ rows of $B$, i.e., the matrix $B$ is equivalent to
$\widetilde{A}$ over $\mathbb{Z}$.  

(b) It suffices to prove the inclusion $\mathbb{R}_+{\mathcal
A}\cap\mathbb{Z}^n\subset \mathbb{N}{\mathcal A}$. Let $a$ be an integral
vector in $\mathbb{R}_+{\mathcal A}$. Then 
$a=\lambda_1v_1+\cdots+\lambda_qv_q$, $\lambda_i\geq 0$ for all $i$.
Set $b=\sum_i\lambda_i$ and denote the {\it ceiling\/} of $b$ by $\lceil
b\rceil$. Recall that $\lceil b\rceil=b$ if $b\in\mathbb{N}$ and 
$\lceil b\rceil=\lfloor
b\rfloor+1$ if $b\not\in\mathbb{N}$, where $\lfloor b\rfloor$ 
is the integer part of $b$. Then $|a|=bd$. We claim that 
$(a,\lceil b\rceil)$ belongs to $\mathbb{R}_+{\mathcal A}'$, where
$\mathcal{A}'$ is the set 
$\{e_1,\ldots,e_n,(v_1,1),\ldots,(v_q,1)\}$.
Let $C_1,\ldots,C_s$ be the
minimal vertex covers of $\mathcal C$ and let $u_i$ be the incidence
vector of $C_i$ for $1\leq i\leq s$. Since $Q(A)$ is integral, by
Theorem~\ref{intiffint}, we can
write  
\begin{equation}\label{aug28-06} 
{\mathbb R}_+{\mathcal A}'=H_{e_1}^+\cap H_{e_2}^+\cap\cdots\cap
H_{e_{n+1}}^+\cap H_{\ell_1}^+\cap H_{\ell_2}^+\cap\cdots\cap 
H_{\ell_s}^+,
\end{equation}
where $\ell_i=(u_i,-1)$ for $1\leq i\leq s$. Notice that 
$(a,b)\in\mathbb{R}_+{\mathcal A}'$, thus using Eq.~(\ref{aug28-06}) we
get that $\langle a,u_i\rangle\geq b$ for all $i$. Hence 
$\langle a,u_i\rangle\geq \lceil b\rceil$ for all $i$ because
$\langle a,u_i\rangle$ is an integer for all $i$. 
Using Eq.~(\ref{aug28-06}) again we get that $(a,\lceil
b\rceil)\in\mathbb{R}_+{\mathcal A}'$, as claimed.  By
Theorem~\ref{noclu1} the Rees ring $R[It]$ is normal. Consequently
applying Lemma~\ref{dec20-07}, we obtain that 
$(a,\lceil b\rceil)\in\mathbb{N}{\mathcal A}'$. There are 
non-negative integers $\eta_1,\ldots,\eta_q$ and
$\rho_1,\ldots,\rho_n$ such that 
$$
(a,\lceil
b\rceil)=\eta_1(v_1,1)+\cdots+\eta_q(v_q,1)+\rho_1e_1+\cdots+\rho_ne_n.
$$
Hence it is seen that 
$|a|=\lceil b\rceil d+\sum_i\rho_i=bd$. Consequently $\rho_i=0$ for
all $i$ and $b=\lceil b\rceil$. It follows at once that
$a\in\mathbb{N}{\mathcal A}$ as required. \end{proof}

The next example shows that the uniformity hypothesis is essential in
the two statements of Theorem~\ref{mfmch->diag}. 

\begin{example}\label{unif-ess-mfmch}\rm Consider 
the clutter $\mathcal{C}$ whose incidence matrix is 
$$
A=\left[\begin{array}{cccc}
1 &0 &0 &1\\ 
0 &1& 0& 1\\ 
0& 0& 1& 1\\ 
0& 1& 1& 0\\ 
1 &0 &1 &0
\end{array}
\right].
$$
Let $v_1,v_2,v_3,v_4$ be the columns of $A$. This clutter is not
uniform, satisfies max-flow min-cut, $A$ is not equivalent over
$\mathbb{Z}$ to an identity matrix, and $\{v_1,\ldots,v_4\}$ is not a Hilbert
basis for the cone it generates.
\end{example}

\begin{corollary}\label{mfmc->smith=identity} 
If $\mathcal{C}$ is a uniform clutter with the
max-flow min-cut property, then its incidence matrix diagonalizes 
over $\mathbb{Z}$ to an identity matrix.
\end{corollary}

\begin{proof}  By Theorem~\ref{mfmch->diag} one has $\Delta_r(A)=1$,
where $r$ is the 
rank of $A$. Thus the invariant factors of $A$ are all equal to $1$
(see \cite[Theorem~3.9]{JacI}), i.e., the Smith normal 
form of $A$ is an identity matrix. \end{proof}  

\begin{corollary}\label{dec14-07} Let $\mathcal{C}$ be a uniform
clutter. Then the following are equivalent{\rm :}
\begin{enumerate}
\item[(i)]  $\mathcal{C}$ has the max-flow min-cut property.
\item[(ii)] $Q(A)$ is an integral polyhedron and 
$\mathbb{N}{\mathcal A}=\mathbb{R}_+{\mathcal
A}\cap\mathbb{Z}^n$, where ${\mathcal A}=\{v_1,\ldots,v_q\}$.
\end{enumerate}
\end{corollary}

\begin{proof}  By Theorems~\ref{noclu1} and \ref{mfmch->diag} we obtain that
(i) implies (ii). Next we prove that (ii) implies (i). By
Proposition~\ref{two-char}(i) it suffices to prove that $K[Ft]=A(P)$.
Clearly $K[Ft]\subset A(P)$. To show the other
inclusion take $x^at^b\in
A(P)$, i.e., $a\in bP\cap\mathbb{Z}^n$. Then from 
the equality 
$\mathbb{N}{\mathcal A}=\mathbb{R}_+{\mathcal A}\cap\mathbb{Z}^n$ it is seen
that $a=\eta_1v_1+\cdots+\eta_qv_q$ for some $\eta_i$'s in $\mathbb{N}$
such that $\sum_i\eta_i=b$. Thus $x^at^b\in K[Ft]$, as required. \end{proof}  

\begin{corollary}\label{pm-iff-n=gd} Let $\mathcal{C}$ be a
$d$-uniform clutter with $n$ vertices. If $\mathcal{C}$ has 
the max-flow min-cut property, then $\mathcal{C}$ has a perfect
matching if and only if $n=d\alpha_0(\mathcal{C})$.
\end{corollary}

\begin{proof}  $\Rightarrow$) As $Q(A)$ is integral and $\mathcal{C}$
has a perfect 
matching, from the proof of
Proposition~\ref{integral+uniforme+pm} we obtain the equality 
$n=d\alpha_0(\mathcal{C})$.

$\Leftarrow$) We set $g=\alpha_0(\mathcal{C})$. Let $u_1,\ldots,u_s$
be the characteristic vectors of the
minimal vertex covers of $\mathcal C$ and let $B$ be the matrix
with column vectors $u_1,\ldots,u_s$. Then
$\langle\mathbf{1},u_i\rangle\geq g$ for all $i$ because any minimal
vertex  
cover of $\mathcal{C}$ has at least $g$ vertices. Thus the 
vector $\mathbf{1}/g$ belongs to the polyhedron $Q(B)=\{x\vert
x\geq0;xB\geq\mathbf{1}\}$. As $Q(A)$ is integral, by 
\cite[Theorem~1.17]{cornu-book} we get that $Q(B)$ is an integral
polyhedron. Consequently
$$ 
Q(B)=\mathbb{R}_+^n+{\rm conv}(v_1,\ldots,v_q),
$$
where $\mathcal{A}=\{v_1,\ldots,v_q\}$ is the set of column vectors 
of the incidence matrix of
$\mathcal{C}$. Therefore since the rational vector $\mathbf{1}/g$ is
in $Q(B)$ we can write
$$
\mathbf{1}/g=\delta+\mu_1v_1+\cdots+\mu_qv_q\ \ \
(\delta\in\mathbb{R}_+^n;\, \mu_i\geq 0;\, \mu_1+\cdots+\mu_q=1).
$$
Hence $n/g=|\delta|+(\sum_{i=1}^q\mu_i)d=|\delta|+d$. Since $n=dg$, we
obtain that $\delta=0$. Thus the vector $\mathbf{1}$ is in 
$\mathbb{R}_+\mathcal{A}\cap \mathbb{Z}^n$. By
Theorem~\ref{mfmch->diag}(b) this intersection is equal to 
$\mathbb{N}\mathcal{A}$. Then we can write
$\mathbf{1}=\eta_1v_1+\cdots+\eta_qv_q$, for some
$\eta_1,\ldots,\eta_q$ 
in $\mathbb{N}$. Hence it is readily seen that $\mathcal{C}$ has a
perfect matching. \end{proof}  

Packing problems of clutters occur in many contexts of combinatorial
optimization \cite{cornu-book,Schr2}, especially where the question
of whether a linear program and its dual have integral optimum solutions is
fundamental.  

\begin{definition}\label{pp-def}\rm 
A clutter $\mathcal C$ satisfies the {\it packing
property\/} 
if all its minors satisfy the K\"onig property, i.e., 
$\alpha_0({\mathcal C}')=\beta_1({\mathcal C}')$ 
for any minor ${\mathcal C}'$ of $\mathcal C$.
\end{definition}

To study linear algebra properties and ring theoretical properties 
of uniform clutters with the packing property we need 
the following result. This interesting result of Lehman is essential in the
proofs of Theorem~\ref{jan18-04} and Corollary~\ref{may28-06} because
it allows to use the structure theorems presented in
Section~\ref{section2}, it also allows to state some conjectures 
about the packing property. 

\begin{theorem}{\rm(A. Lehman \cite{lehman}; see \cite[Theorem~1.8]{cornu-book})}
\label{lehman} If\, a clutter $\mathcal C$ has the packing property, then $Q(A)$ is
integral.
\end{theorem}

\begin{proposition}{\rm\cite{cornu-book}}\label{aug25-06} If a clutter
$\mathcal C$ has the max-flow
min-cut property,  
then $\mathcal C$ has the packing property. 
\end{proposition}

\begin{proof}  It suffices to prove that $\mathcal{C}$ has the K\"onig property
because the max-flow min-cut property is closed under taking minors.
Making $\alpha=\mathbf{1}$ in Eq.~(\ref{jun6-2-03-1}), we get
that the LP-duality equation:
$$
{\rm min}\{\langle\mathbf{1},x\rangle \vert\, x\geq 0; xA\geq \mathbf{1}\}=
{\rm max}\{\langle y,\mathbf{1}\rangle \vert\, y\geq 0; Ay\leq\mathbf{1}\}
$$
has optimum integral solutions $x$, $y$. To complete the 
proof notice that the left hand side of this 
equality is $\alpha_0({\mathcal C})$ and the right hand side 
is $\beta_1({\mathcal C})$. \end{proof}  

Conforti and Cornu\'ejols conjecture that the converse is also
true:

\begin{conjecture}{\rm(\cite{CC},
\cite[Conjecture~1.6]{cornu-book})}\label{conforti-cornuejols1} 
\rm  If a clutter $\mathcal C$ satisfies the
packing property, then $\mathcal C$ has the max-flow min-cut property.
\end{conjecture} 

To the best of our knowledge this conjecture is open. For uniform
clutters, using Proposition~\ref{two-char} 
and Theorem~\ref{lehman}, we obtain the following  
algebraic version of this conjecture:

\begin{conjecture}\label{con-cor-vila}\rm If $\mathcal C$ is a uniform
clutter with the packing property, then the homogeneous ring
$K[Ft]$ equals the Ehrhart ring $A(P)$. 
\end{conjecture} 

Conjecture~\ref{conforti-cornuejols1} together with 
Proposition~\ref{two-char} 
suggest the following: 

\begin{conjecture}{\rm \cite{reesclu}}
\label{git-val-vi}\rm\ If $\mathcal C$ is a uniform clutter with the
packing property, then any of the following equivalent conditions
hold: 
\begin{enumerate}
\item[(a)] $\mathbb{Z}^{n+1}/((v_1,1),\ldots,(v_q,1))$ is a free
group.
\item[(b)] $\Delta_r(B)=1$,  where $B$ is the
matrix with column vectors $(v_1,1),\ldots,(v_q,1)$ and $r$ is the 
rank of $B$. 
\item[(c)] $B$ diagonalizes over $\mathbb{Z}$ to an identity
matrix.
\item[(d)] $\overline{K[Ft]}=A(P)$, where $A(P)$ is the Ehrhart
ring of $P={\rm conv}(v_1,\ldots,v_q)$. 
\end{enumerate}
\end{conjecture} 

This conjecture will be proved in
Section~\ref{section-packing-partition} 
for uniform clutters with a perfect matching
and vertex covering number equal to $2$ (see Theorem~\ref{jan18-04}). 

\section{Minors and the packing
property}\label{section-packing-partition}
Let $\mathcal{C}$ be a $d$-uniform clutter with vertex set
$X=\{x_1,\ldots,x_n\}$, let
$x^{v_1},\ldots,x^{v_q}$ be the minimal set 
of generators of the edge ideal $I=I(\mathcal{C})$, and let $A$ be
the incidence matrix of $\mathcal{C}$ with column vectors 
$v_1,\ldots,v_q$. 
We denote the transpose of $A$ by $A^t$. All clutters considered in
this section have vertex covering number equal to $2$. We shall be
interested in 
studying the relationships between 
the combinatorics of $\mathcal{C}$, the algebraic properties 
of the Ehrhart ring $A(P)$, the linear algebra of the incidence
matrix $A$, and the algebra of the edge ideal $I(\mathcal{C})$.

We begin by showing that nice combinatorial properties 
of $\mathcal{C}$ are reflected in nice linear algebra properties of
$A$. The following is one of the main results of this section. 
It gives some support to Conjecture~\ref{git-val-vi}.

\begin{theorem}\label{jan18-04} Let $\mathcal{C}$ be a 
$d$-uniform clutter with a perfect matching such that
$\mathcal{C}$ has the packing property and $\alpha_0(\mathcal{C})=2$.
If $A$ has rank $r$, 
then 
$$
\Delta_r\hspace{-1mm}\left(\hspace{-2mm}
\begin{array}{c}A\\ 
\mathbf{1}\end{array}\hspace{-2mm}\right)=1.
$$
\end{theorem}

\begin{proof}  By the Lehman theorem the polyhedron $Q(A)$ is
integral (see Theorem~\ref{lehman}). Thus by  
Proposition~\ref{integral+uniforme+pm-cor} there is a perfect
matching $f_1,f_2$ of $\mathcal{C}$ with $X=f_1\cup f_2$ and there is a
partition $X_1,\ldots,X_d$ of $X$ such that  $X_i$ is a minimal vertex
cover of $\mathcal{C}$ for all $i$, $|X_i|=2$ for all $i$, and  
\begin{equation}\label{jan19-04}
|{\rm supp}(x^{v_i})\cap X_k|=1\ \ \ \ \forall\ i,k.
\end{equation}
Thus we may assume that $X_i=\{x_{2i-1},x_{2i}\}$ for $i=1,\ldots,d$.
Notice that 
$n=2d$ and ${\rm rank}(A)\geq 2$ because $X=f_1\cup f_2$.

We proceed by induction on $r={\rm rank}(A)$. Since the sum of the first
two rows of $A$ is equal to $\mathbf 1$, it suffices to
prove that $1$ is the only invariant factor of $A^t$ or equivalently
that the Smith normal form of $A^t$ is the identity. 

First we show the case $r=2$ which is the base case for the induction
process. We may assume that $v_1$ and $v_2$ are the characteristic
vectors of $f_1$ and $f_2$ respectively. The matrix $A_0$ with rows
$v_1,v_2$ is equivalent over $\mathbb{Z}$ to 
$$
\left(\hspace{-1mm}
\begin{array}{ccccc}
1&0&0&\cdots&0\\ 
0&1&0&\cdots &0
\end{array}\hspace{-1mm}
\right)
$$
and this is the Smith normal form of $A_0$. 
Thus the quotient group $M=\mathbb{Z}^n/\mathbb{Z}\{v_1,v_2\}$ is
torsion-free, 
where $\mathbb{Z}\{v_1,v_2\}$ is the subgroup of $\mathbb{Z}^n$
generated by $\{v_1,v_2\}$.
We claim that
$\mathbb{Z}\{v_1,v_2\}$ is equal to
$\mathbb{Z}\{v_1,v_2,\ldots,v_q\}$. Since the 
rank of $A$ is $2$, for each $i\geq 3$ there is an integer
$0\neq\eta_i$ such that $\eta_iv_i\in \mathbb{Z}\{v_1,v_2\}$. Hence
the image $\overline{v}_i$ of $v_i$ in $M$ is a torsion element, so
$\overline{v}_i$ must be zero, that is, 
$v_i\in\mathbb{Z}\{v_1,v_2\}$. This completes the proof of the claim.
Therefore
$$
\Delta_2(A^t)=|T(\mathbb{Z}^n/\mathbb{Z}\{v_1,\ldots,v_q\})|=
|T(\mathbb{Z}^n/\mathbb{Z}\{v_1,v_2\})|=1,
$$
where $T(\mathbb{Z}^n/\mathbb{Z}\{v_1,\ldots,v_q\})$ is the torsion
subgroup of $\mathbb{Z}^n/\mathbb{Z}\{v_1,\ldots,v_q\}$, 
the first equality follows from \cite[Theorem~3.9]{JacI} and 
\cite[pp.~187-188]{JacI}. 
Therefore $\Delta_2(A^t)=1$ and the Smith normal form of 
$A^t$ is the identity, as required.

We continue with the induction process by assuming $r\geq 3$. Let
$w_1,\ldots,w_{2d}$ be the columns of $A^t$ and let $V_i$ be the
linear space generated by $w_1,\ldots,w_{2i}$. Notice that, because of
Eq.~(\ref{jan19-04}), for each
odd integer 
$k$ the sum of rows $k$ and $k+1$ of the matrix
$A$ is equal to ${\mathbf 1}=(1,\ldots,1)$, i.e.,
$w_k+w_{k+1}=\mathbf{1}$ for $k$ odd. Thus if $k$ is odd and
we remove columns 
$w_k$ and $w_{k+1}$ from $A^t$ we
obtain a submatrix whose rank is greater than or equal to $r-1$. 
Thus after permuting columns we may assume 
\begin{equation}\label{jan20-04}
\dim(V_i)=\left\{\begin{array}{lll} i+1& \mbox{if}& 1\leq i\leq r-1,\\
r&\mbox{if}& r\leq i.
\end{array}\right.
\end{equation}
Let $J$ be the square-free monomial ideal defined by the rows of the matrix 
$[w_1,\ldots,w_{2(r-1)}]$, where $w_1,\ldots,w_{2(r-1)}$ are column
vectors, and let $\mathcal{D}$ be the clutter associated to the edge
ideal $J$. If $v_i=(v_{i,1},\ldots,v_{i,n})$ for $1\leq i\leq q$, then $J$ 
is generated by the monomials: 
$$
x_1^{v_{1,1}}x_2^{v_{1,2}}\cdots
x_{2(r-1)}^{v_{1,2(r-1)}},\ldots,x_1^{v_{i,1}}
x_2^{v_{i,2}}\cdots
x_{2(r-1)}^{v_{i,2(r-1)}},\ldots,x_1^{v_{q,1}}x_2^{v_{q,2}}\cdots
x_{2(r-1)}^{v_{q,2(r-1)}}.
$$
Thus $J$ is a minor of $I$ because $J$ is obtained from $I$ by making
$x_i=1$ for $i>2(r-1)$. Then a
minimal set of generators of $J$ consists of monomials of degree
$r-1$, $\alpha_0(\mathcal{D})=2$, and $\mathcal{D}$ has a perfect
matching.  Furthermore $\mathcal{D}$ satisfies the
packing property because $J$ is a minor of $I$. If
$[w_1,\ldots,w_{2(r-1)}]$ diagonalizes (over 
the integers) to the identity matrix, so does $A^t$, this follows 
using arguments similar to those used for the case ${\rm rank}(A)=2$. 
Therefore we may harmlessly assume $d=r-1$, $I=J$, and $\mathcal
{C}=\mathcal{D}$. This means that the matrix $A$ has rank $d+1$, the
maximum possible.

Let $B$ be the matrix $[w_1,\ldots,w_{2(d-1)}]$ and let $I'$ be the
monomial ideal defined by the rows of $B$, that is $I'$ is obtained
from $I$ making $x_{2d-1}=x_{2d}=1$. The matrix $B$ has rank $r-1$.
Hence by 
induction hypothesis $B$ diagonalizes to a 
matrix $[I_{r-1},\mathbf{0}]$, where $I_{r-1}$ is the identity matrix
of order $r-1$. Recall that $f_1,f_2$
is a perfect matching of $\mathcal{C}$. Then by permuting rows and
columns we may assume that the matrix $A^t$ is written as:
\begin{small}
$$
\begin{array}{l}
\begin{array}{ccccccc}
1 0& 1 0& 1 0&\cdots& 1 0& 1 0& \leftarrow\\ 
0 1& 0 1& 0 1&\cdots& 0 1& 0 1&\\
\bigcirc &\bigcirc&\bigcirc&\cdots&\bigcirc & 1 0&\leftarrow \\
\vdots &\vdots & & &\vdots&\vdots&\vdots  \\ 
\bigcirc &\bigcirc&\bigcirc&\cdots&\bigcirc & 1 0&\leftarrow  \\ 
\bigcirc &\bigcirc&\bigcirc&\cdots&\bigcirc & 0 1&\\
\vdots &\vdots & & &\vdots&\vdots\\
\bigcirc &\bigcirc&\bigcirc&\cdots&\bigcirc & 0 1&
\vspace{-3mm}
\end{array}\\
\underbrace{\ \ \ \ \ \ \ \ \ \ \ \ \ \ \ \ \ \ \ \ \ 
\ \ \ \ \ \ \ }\\
\ \ \ \ \ \ \ \ \ \ \ \ \ B
\end{array}
$$
\end{small}
\noindent where either a pair $1\, 0$ or $0\, 1$ must occur in the
places marked with a 
circle and such that the number of $1's$ in the last column is
greater than or equal to the number of $1's$ in any other column. Consider the
matrix $C$ obtained from $A^t$ by removing the rows whose penultimate
entry is equal to $1$ (these are marked above with an arrow) and
removing the last column. Let $K$ be the monomial ideal defined by the rows of
$C$, that is $K$ is obtained from $I$ by making $x_{2d-1}=0$ and
$x_{2d}=1$. By the choice
of the last column and because of Eq.~(\ref{jan20-04}) it is seen 
that $K$ has height two. Since $K$ is a minor of $I$ it has the
K\"onig property. Consequently the matrix $A^t$ has one of the
following two forms:
\begin{small}
$$
\begin{array}{cc}
\begin{array}{l}
\begin{array}{ccccccc}
1 0& 1 0& 1 0&\cdots& 1 0& 1 0& \\ 
0 1& 0 1& 0 1&\cdots& 0 1& 0 1&\\
\bigcirc &\bigcirc&\bigcirc&\cdots&\bigcirc & 1 0& \\
\vdots &\vdots & & &\vdots&\vdots&  \\ 
\bigcirc &\bigcirc&\bigcirc&\cdots&\bigcirc & 1 0&  \\ 
1 0&1 0&1 0&\cdots&1 0 & 0 1&\\
\bigcirc &\bigcirc&\bigcirc&\cdots&\bigcirc & 0 1&\\
\vdots &\vdots & & &\vdots&\vdots\\
\bigcirc &\bigcirc&\bigcirc&\cdots&\bigcirc & 0 1&
\vspace{-3mm}
\end{array}\\
\underbrace{\ \ \ \ \ \ \ \ \ \ \ \ \ \ \ \ \ \ \ \ \ 
\ \ \ \ \ \ \ }\\
\ \ \ \ \ \ \ \ \ \ \ \ \ B
\end{array}
& 
\begin{array}{l}
\begin{array}{ccccccl}
1 0& 1 0& 1 0&\cdots& 1 0& 1 0& \mbox{row}\ v_1 \\ 
0 1& 0 1& 0 1&\cdots& 0 1& 0 1& \mbox{row}\ v_2\\
\bigcirc &\bigcirc&\bigcirc&\cdots&\bigcirc & 1 0& \\
\vdots &\vdots & & &\vdots&\vdots&  \\ 
\bigcirc &\bigcirc&\bigcirc&\cdots&\bigcirc & 1 0&  \\ 
\bigcirc &\bigcirc&\bigcirc&\cdots&\bigcirc & 0 1& \mbox{row}\ v_j\\ 
\bigcirc &\bigcirc&\bigcirc&\cdots&\bigcirc & 0 1&\mbox{row}\ v_{j+1}\\ 
\vdots &\vdots & & &\vdots&\vdots\\
\bigcirc &\bigcirc&\bigcirc&\cdots&\bigcirc & 0 1&
\vspace{-3mm}
\end{array}\\
\underbrace{\ \ \ \ \ \ \ \ \ \ \ \ \ \ \ \ \ \ \ \ \ 
\ \ \ \ \ \ \ }\\
\ \ \ \ \ \ \ \ \ \ \ \ \ B
\end{array}
\end{array}
$$
\end{small}
\noindent where in the second case one has
$v_j+v_{j+1}=(1,1,\ldots,1,0,2)$. In the second case, using row
operations, we may replace $v_2$ by $v_j+v_{j+1}-v_2$. In the first
case by permuting rows $v_2$ and $v_j$ we may replace $v_2$ by
$(1,0,\ldots,1,0,0,1)$. Thus in both cases, using
row operations, we get that $A^t$ can be brought to the form: 
\begin{small}
$$
\begin{array}{l}
\begin{array}{ccccccc}
1 0& 1 0& 1 0&\cdots& 1 0& 1 0&  \\ 
1 0& 1 0& 1 0&\cdots& 1 0& 0 1&\\
\bigcirc &\bigcirc&\bigcirc&\cdots&\bigcirc & 1 0&  \\
\vdots &\vdots & & &\vdots&\vdots&  \\ 
\bigcirc &\bigcirc&\bigcirc&\cdots&\bigcirc & 1 0&   \\ 
\bigcirc &\bigcirc&\bigcirc&\cdots&\bigcirc & 0 1&\\
\vdots &\vdots & & &\vdots&\vdots\\
\bigcirc &\bigcirc&\bigcirc&\cdots&\bigcirc & 0 1&
\vspace{-3mm}
\end{array}\\
\underbrace{\ \ \ \ \ \ \ \ \ \ \ \ \ \ \ \ \ \ \ \ \ 
\ \ \ \ \ \ \ }\\
\ \ \ \ \ \ \ \ \ \ \ \ \ B_1
\end{array}
$$
\end{small}
\noindent where $B_1$ has rank $r-1$ and diagonalizes to an identity.
Therefore it is readily seen that this matrix is equivalent to 

\begin{small}
$$
\left[\begin{array}{ccccccr}
I_{r-1}&0 & 0 &\cdots& 0 & 0 & 0  \\ 
\mathbf{0}& 0 & 0&\cdots&  0& 1 &-1\\
\mathbf{0}&0 &0 &\cdots&0 &a_1&b_1  \\
\vdots &\vdots & & &\vdots&\vdots&\vdots  \\ 
\mathbf{0}&0 &0 &\cdots&0 &a_{s}&b_{s}
\end{array}\right]
$$
\end{small}

\noindent for some integers $a_1,b_1,\ldots,a_s,b_s$. Next 
for $1\leq i\leq s$ we multiply the second
row by $-a_i$ and add it to row $i+2$. Then we add the last two
columns. Using these row and column operations this matrix can be 
brought to the 
form:

\begin{small}
$$
\left[\begin{array}{ccccccr}
I_{r-1}&0 & 0 &\cdots& 0 & 0 & 0  \\ 
\mathbf{0}& 0 & 0&\cdots&  0& 1 &0\\
\mathbf{0}&0 &0 &\cdots&0 &0&c_1  \\
\vdots &\vdots & & &\vdots&\vdots&\vdots  \\ 
\mathbf{0}&0 &0 &\cdots&0 &0&c_{s} \\
\end{array}\right]
$$
\end{small}

\noindent for some integers $c_1,\ldots,c_s$. 
To finish the proof observe that $c_i=0$ for all $i$ 
because this matrix has rank $r$. Hence 
this matrix reduces to $[I_r,\mathbf{0}]$, i.e., $A^t$ is equivalent
to the identity matrix $[I_r,\mathbf{0}]$, 
as required. \end{proof}

Next we show that clutters with nice algebraic and combinatorial properties 
have nice combinatorial optimization properties.  

\begin{corollary}\label{may28-06} Let $\mathcal{C}$ be a 
$d$-uniform clutter with a perfect matching such that 
$\mathcal{C}$ has the packing property and $\alpha_0(\mathcal{C})=2$.
If $K[Ft]$ is normal, then $\mathcal C$ has the max-flow min-cut property.
\end{corollary}

\begin{proof} Let $B$ be the matrix with column vectors
$(v_1,1),\ldots,(v_q,1)$. By Theorem~\ref{jan18-04} we have
$\Delta_r(B)=1$, where $r={\rm rank}(A)$. Notice that $r={\rm
rank}(B)$ because $\mathcal{C}$ is uniform. According
to Proposition~\ref{two-char}(ii) the condition $\Delta_r(B)=1$ is
equivalent to the 
equality
$\overline{K[Ft]}=A(P)$. Thus by the normality of $K[Ft]$ we get 
$$
K[Ft]=\overline{K[Ft]}=A(P).
$$
By the Lehman theorem the polyhedron $Q(A)$ is
integral (see Theorem~\ref{lehman}). 
Hence by Proposition \ref{two-char}(i) 
we get that $\mathcal C$ satisfies the max-flow min-cut property. 
\end{proof}  

The next result gives some support to Conjecture~\ref{conforti-cornuejols1}. 

\begin{corollary}\label{cc-conj-supp} Let $\mathcal{C}$ be a 
$d$-uniform clutter with a perfect matching such that 
$\mathcal{C}$ has the packing property and $\alpha_0(\mathcal{C})=2$.
If $v_1,\ldots,v_q$ are linearly independent, then $\mathcal C$ has the
max-flow 
min-cut property.
\end{corollary}

\begin{proof} It follows at once from Corollary~\ref{may28-06} because
$K[Ft]$ is normal. \end{proof}   

A $d$-uniform clutter $\mathcal{C}$ is called $2$-{\it partitionable} 
if its vertex set $X$ has a partition $X_1,\ldots,X_d$ such that 
$X_i$ is a minimal vertex cover of $\mathcal{C}$ and $|X_i|=2$ for
$i=1,\ldots,d$. For instance the clutters of Theorem~\ref{jan18-04}
are $2$-partitionable. Another instance is the clutter of minimal
vertex covers of an unmixed bipartite graph (see
Corollary~\ref{may17-09}). 
A clutter is called {\it unmixed}
if all its minimal vertex covers have the same size. A specific
illustration of a famous  
$2$-partitionable clutter is given in Example~\ref{q6}.  
Our next result is about this family of clutters. 

Recall that an edge ideal $I\subset R$ is called {\it normal} if
$\overline{I^i}=I^i$ for  
$i\geq 1$, where
$$
\overline{I^i}=(\{x^a\in R\vert\, \exists\, p\geq 1;(x^a)^{p}\in
I^{pi}\})\subset R
$$ 
is the integral closure of $I^i$. Also recall that 
$I$ is called {\it minimally non-normal\/} if $I$ is not normal
and all its proper minors are normal. An edge ideal $I$ is normal if and
only if its Rees algebra $R[It]$ is normal (see
\cite[Theorem~3.3.18]{monalg}). 

The other main result of this section is:

\begin{theorem}\label{another-mt} Let $\mathcal{C}$ be a $d$-uniform
clutter with a 
partition $X_1,\ldots,X_d$ of $X$ such that $X_i=\{x_{2i-1},x_{2i}\}$
is a minimal vertex
cover of $\mathcal{C}$ for all $i$. Then
\begin{enumerate}
\item[(a)] ${\rm rank}(A)\leq d+1$.
\item[(b)] If $C$ is a minimal vertex cover of $\mathcal{C}$, 
then $2\leq|C|\leq d$. 
\item[(c)] If $\mathcal C$ satisfies the K\"onig property and there is
a minimal vertex
cover $C$ of $\mathcal C$ with $|C|=d\geq 3$, then ${\rm
rank}(A)=d+1$.
\item[(d)] If $I=I(\mathcal{C})$ is minimally non-normal and
$\mathcal{C}$ satisfies the packing property, then ${\rm rank}(A)=d+1$.
\end{enumerate}
\end{theorem}

\begin{proof}  (a) For each odd integer $k$ the sum of rows $k$ and
$k+1$ of the matrix 
$A$ is equal to ${\mathbf 1}=(1,\ldots,1)$. Thus the rank of $A$ is
bounded by $d+1$.

(b) By the pigeon hole principle, any minimal vertex cover $C$ of the
clutter $\mathcal C$ 
satisfies $2\leq|C|\leq d$. 

(c) First notice that $C$ contains exactly one element of each $X_j$
because $X_j\not\subset C$ and $|C|=d$. Thus we may assume
$$
C=\{x_1,x_3,\ldots,x_{2d-1}\}.
$$
Consider the monomial $x^\alpha=x_2x_4\cdots x_{2d}$ 
and notice that $x_kx^\alpha\in I$ for each $x_k\in C$ because the
monomial $x_kx^\alpha$ is clearly in every minimal prime of $I$.
Writing $x_k=x_{2i-1}$ with $1\leq i\leq d$ we conclude that the
monomial
\begin{equation}\label{may15-09}
x^{\alpha_i}=\frac{x_{2i-1}x^\alpha}{x_{2i}}=
\frac{x_{2i-1}(x_2x_4\cdots x_{2d})}{x_{2i}}
\end{equation}
is a minimal generator of $I$. Thus we may assume
$x^{\alpha_i}=x^{v_i}$ for $i=1,\ldots,d$. The vector $\mathbf 1$
belongs to the linear space generated by $v_1,\ldots,v_q$ because
$\mathcal C$ has the K\"onig property, where $q$ is the number of
edges of $\mathcal{C}$. By part (a) the rank of $A$ is at most $d+1$. 
Thus it suffices to prove that $v_1,\ldots,v_d,\mathbf{1}$ 
are linearly independent. Assume that 
\begin{equation}\label{may15-09-1}
\lambda_1v_1+\cdots+\lambda_dv_d+\lambda_{d+1}\mathbf{1}=0
\end{equation}
for some scalars $\lambda_1,\ldots,\lambda_{d+1}$. From
Eq.~(\ref{may15-09}) we have 
\begin{equation}\label{may15-09-2}
v_i=e_{2i-1}-e_{2i}+\sum_{j=1}^de_{2j}\ \ \ \ \ 
\end{equation}
for $i=1,\ldots,d$. Hence using Eq.~(\ref{may15-09-1}) 
we conclude that 
$$
\sum_{i=1}^d\lambda_ie_{2i-1}+\lambda_{d+1}\sum_{i=1}^de_{2i-1}=
\sum_{i=1}^d(\lambda_i+\lambda_{d+1})e_{2i-1}=0.
$$
Hence $\lambda_{d+1}=-\lambda_i$ for $i=1,\ldots,d$. Using 
Eq.~(\ref{may15-09-1}) once more we get
$$
\lambda_{d+1}(v_1+\cdots+v_d-\mathbf{1})=0.
 $$
As $v_1+\cdots+v_d-\mathbf{1}$ cannot be zero by Eq.~(\ref{may15-09-2}),
we obtain that $\lambda_{d+1}=0$. Consequently $\lambda_i=0$ for all
$i$, as required.

(d) Let $x^\alpha t^b$ be a minimal generator of 
$\overline{R[It]}$ not in $R[It]$ and let $m=x^\alpha$. Then
$m\in\overline{I^b}\setminus I^b$. Using 
\cite[Proposition~4.3]{reesclu}, one has $\deg(m)=bd$. Consider the $R$-module
$N=\overline{I^b}/I^b$. Notice that the image of $m$ in $N$ is
nonzero. The set of associated primes of $N$, denoted
by ${\rm Ass}(N)$, is the set of prime ideals $\mathfrak{p}$ of
$R$ such that $N$ contains a submodule isomorphic to
$R/\mathfrak{p}$. By the hypothesis that $I$ is minimally non-normal,
we have that $\mathfrak{m}=(x_1,\ldots,x_n)$ is the only associated
prime of $N$. Hence, using 
\cite[Theorem 9, p.~50]{Mat}, we have the following expression for the
radical of the annihilator of $N$
$$
{\rm rad}({\rm ann}(N))=\bigcap_{\mathfrak{p}\in{\rm Ass}(N)}
\mathfrak{p}=\mathfrak{m}=(x_1,\ldots,x_n).
$$
Consequently $\mathfrak{m}^r\subset{\rm ann}(N)=\{x\in R\vert xN=0\}$ 
for some $r>0$. Thus for $i$ odd we can
write 
$$
x_i^rx^\alpha=(x^{v_1})^{a_1}\cdots(x^{v_q})^{a_q}x^\delta,
$$
where $a_1+\cdots+a_q=b$ and $\deg(x^\delta)=r$. If we write 
$x^\delta=x_i^{s_1}x_{i+1}^{s_2}x^\gamma$ with $x_i,x_{i+1}$ not in
the support of $x^{\gamma}$, making $x_j=1$ for $j\notin\{i,i+1\}$, 
it is not hard to see that $r=s_1+s_2$
and $\gamma=0$. Thus we get an equation:
$$
x_i^{s_2}x^\alpha=(x^{v_1})^{a_1}\cdots(x^{v_q})^{a_q}x_{i+1}^{s_2}
$$
with $s_2>0$. Using a similar argument we obtain an equation:
$$
x_{i+1}^{w_1}x^\alpha=(x^{v_1})^{b_1}\cdots(x^{v_q})^{b_q}x_{i}^{w_1}
$$
with $w_1>0$. Therefore 
$$
x_{i+1}^{s_2+w_1}(x^{v_1})^{a_1}\cdots(x^{v_q})^{a_q}
=x_{i}^{s_2+w_1}(x^{v_1})^{b_1}\cdots(x^{v_q})^{b_q}.
$$
Consider the group $\mathbb{Z}^n/\mathbb{Z}{\mathcal A}$, where 
${\mathcal A}=\{v_1,\ldots,v_q\}$.  Since this group is torsion-free, we 
get $e_i-e_{i+1}\in \mathbb{Z}{\mathcal A}$ for $i$ odd. Finally 
to conclude that ${\rm rank(A)}=d+1$ notice 
that $\mathbf{1}\in\mathbb{Z}{\mathcal A}$. \end{proof}  

\begin{example}{\rm \cite[p.~12]{cornu-book}}\label{q6}\rm\ Let
$\mathcal{C}$ be the uniform clutter whose edges are
$$
\{x_1,x_4,x_5\},\{x_1,x_3,x_6\},\{x_2,x_4,x_6\},\{x_2,x_3,x_5\}
$$ 
and let $A$ be its incidence matrix.  Using {\sc Normaliz} \cite{normaliz2}, it 
is not hard to see that the minimal vertex covers of $\mathcal{C}$ are
$$
\begin{array}{llll}
X_1=\{x_1,x_2\},&X_2=\{x_3,x_4\},&X_3=\{x_5,x_6\},&
\\
C_4=\{x_1,x_4,x_5\},&C_5=\{x_1,x_3,x_6\},&
C_6=\{x_2,x_4,x_6\},&C_7=\{x_2,x_3,x_5\}.
\end{array}
$$
This clutter satisfies the hypotheses of
Theorem~\ref{another-mt} with $d=3$ and has minimal vertex covers of sizes $2$
and $3$. Moreover the rank of $A$ is $4$, $\mathcal{C}$ does not
satisfy the K\"onig property and $Q(A)$ is integral.  
\end{example}

For use below recall that a graph $G$ is {\it Cohen-Macaulay} if $R/I(G)$ is
a Cohen-Macaulay ring. Any Cohen-Macaulay graph is unmixed
\cite[p.~169]{monalg}.  

\begin{corollary}\label{may17-09} Let $A$ be the incidence matrix of the 
clutter $\mathcal{C}$ of minimal vertex covers of an unmixed bipartite graph
$G$. Then {\rm (i)} $\mathcal{C}$ is $2$-partitionable. 
{\rm (ii)} ${\rm rank}(A)\leq \alpha_0(G)+1$. {\rm (iii)} ${\rm
rank}(A)=\alpha_0(G)+1$ if $G$ is Cohen-Macaulay.
\end{corollary}

\begin{proof} We set $d=\alpha_0(G)$. Clearly $\mathcal{C}$ is a
$d$-uniform clutter because all minimal vertex covers of $G$ 
have size $d$. By \cite[Theorem~78.1]{Schr2} the polyhedron $Q(A)$ is
integral because $G$ is a bipartite graph and bipartite graphs have
integral set covering polyhedron \cite[Proposition~4.27]{reesclu}. 
Any minimal vertex cover of $\mathcal{C}$ is of the
form $\{x_i,x_j\}$ for some edge $\{x_i,x_j\}$ of $G$. Therefore by
Proposition~\ref{nov20-03} the clutter $\mathcal{C}$ is
$2$-partitionable. Thus we have shown (i). Then using part (a) of
Theorem~\ref{another-mt} we obtain (ii). To prove (iii) 
it suffices to prove that there are minimal vertex 
covers $C_1,\ldots,C_{d+1}$ of $G$ whose characteristic vectors are linearly
independent. This follows by induction on the number of vertices and
using the fact that any Cohen-Macaulay bipartite graph has at least
one vertex of degree $1$ \cite[Theorem~6.4.4]{monalg}. 
\end{proof}

\section{Triangulations and the max-flow min-cut
property}\label{section-trian-balanced}

Let $R=K[x_1,\ldots,x_n]$ be a polynomial ring over a 
field $K$ and let $\mathcal{C}$ be a $d$-uniform clutter with vertex set
$X=\{x_1,\ldots,x_n\}$ and edge ideal $I=I(\mathcal{C})$. In what
follows $F=\{x^{v_1},\ldots,x^{v_q}\}$ will denote the minimal set of
generators of $I$ and ${\mathcal A}$ will denote the set 
$\{v_1,\ldots,v_q\}$. The incidence matrix of $\mathcal{C}$, i.e.,
the $n\times q$ matrix with column vectors $v_1,\ldots,v_q$ will be
denoted by $A$. One can think of the columns of $A$ as points in
$\mathbb{R}^n$. The set ${\mathcal A}$ 
is called a {\it point configuration\/}.

Consider a homomorphism of $K$-algebras:
$$
S=K[t_1,\ldots,t_q]\stackrel{\varphi}{\longrightarrow}K[F]\ \ \ \ 
(t_i\stackrel{\varphi}{\longrightarrow}x^{v_i}),
$$
where $S$ is a polynomial ring. The kernel of $\varphi$, 
denoted by $P$, is called the {\it toric ideal\/} of $K[F]$. 
For the 
rest of this section we assume that $\prec$ is a fixed term 
order for the set of monomials of $S$. We denote the initial ideal of $P$
by ${\rm in}_\prec(P)$. If the choice of a 
term order is clear, we may write ${\rm in}(P)$ instead of 
${\rm in}_\prec(P)$. There is a one to one 
correspondence between simplicial
complexes with vertex set $X$ and squarefree monomial ideals of $R$
via the Stanley-Reisner theory \cite{Sta2}:
$$
\Delta\longmapsto I_\Delta=(x^a\vert\,\, x^a\mbox{ is squarefree and }
{\rm supp}(x^a)\notin\Delta).
$$
Notice that ${\rm rad}({\rm in }(P))$, the radical of ${\rm in }(P)$,
is a squarefree monomial ideal. Let $\Delta$ be the simplicial
complex 
whose Stanley-Reisner 
ideal is ${\rm rad}({\rm in }(P))$. 

Let
$\omega=(\omega_i) \in\mathbb{N}^q$ be an integral weight vector. If
$$
f=f(t_1,\ldots,t_q)=\lambda_1t^{a_1}+\cdots+\lambda_st^{a_s}
$$ 
is a polynomial with $\lambda_1,\ldots,\lambda_s$ in $K$, 
we define ${\rm in}_\omega(f)$, 
the {\it initial form\/} of $f$ 
relative to $\omega$, as the sum 
of all terms $\lambda_it^{a_i}$ such that 
$\langle \omega,a_i\rangle$ is maximal. 
The ideal generated by all initial forms 
is denoted by ${\rm in}_\omega(P)$.  

\begin{proposition}{\rm
\cite[Proposition~1.11]{Stur1}}\label{jun4-02} For every
term order $\prec$, ${\rm in}_\prec(P)={\rm
in}_\omega(P)$ for 
some non-negative integer weight vector $\omega\in\mathbb{N}^q$.
\end{proposition}

\begin{theorem}{\rm\cite[Theorem~8.3]{Stur1}}
\label{jun6-02} If ${\rm in}(P)={\rm in}_\omega(P)$, 
then
$$
\Delta=\{\sigma\vert\, \exists\, c\in\mathbb{R}^n
\mbox{ such that }\langle v_i,c\rangle=\omega_i 
\mbox{ if }t_i\in\sigma\, \mbox{ and }\, 
\langle v_i,c\rangle<\omega_i\mbox{ if }t_i\notin\sigma\}.
$$
\end{theorem}

Let $\omega=(\omega_i)\in\mathbb{N}^q$ be a vector that represents
the initial ideal ${\rm in}(P)={\rm in}_\prec(P)$ with respect to a
term order $\prec$, 
that is, ${\rm in}(P)={\rm
in}_\omega(P)$. Consider the primary decomposition of ${\rm rad}({\rm
in}(P))$ as 
an intersection of face ideals:
$$ 
{\rm rad}({\rm
in}(P))=\mathfrak{p}_1\cap\mathfrak{p}_2\cap\cdots\cap\mathfrak{p}_r,
$$
where $\mathfrak{p}_1,\ldots,\mathfrak{p}_r$ are ideals of 
$K[t_1,\ldots,t_q]$ generated by subsets of $\{t_1,\ldots,t_q\}$. 
Recall that the facets of $\Delta$ are given by
$$
F_i=\{t_j\vert\, t_j\notin\mathfrak{p}_i\},
$$
or equivalently by 
${\mathcal A}_i=\{v_j\vert\, t_j\notin\mathfrak{p}_i\}$ if one identifies
$t_i$ with $v_i$. According to Theorem~\ref{jun6-02}, the family
of facets (resp. cones or polytopes) of $\Delta$: 
$$
\{{\mathcal A}_1,\ldots,{\mathcal A}_r\} \ \ \
({\rm resp.}\ \{\mathbb{R}_+{\mathcal A}_1,\ldots,\mathbb{R}_+{\mathcal
A}_r\}\mbox{ or } \{{\rm conv}({\mathcal A}_1),\ldots,{\rm conv}({\mathcal
A}_r)\}),
$$
form a regular triangulation of ${\mathcal A}$ (resp. $\mathbb{R}_+{\mathcal
A}$ or 
${\rm conv}({\mathcal A})$). This means that 
$$
{\rm conv}({\mathcal A}_1),\ldots,{\rm conv}({\mathcal
A}_r)$$ 
are obtained by projection onto 
the first $n$ coordinates of the lower facets of 
$$
Q'={\rm conv}((v_1,\omega_1),\ldots,(v_q,\omega_q)).
$$
The regular triangulation $\{{\mathcal A}_1,\ldots,{\mathcal A}_r\}$ is called
{\it unimodular\/} if $\mathbb{Z}{\mathcal A}_i=\mathbb{Z}{\mathcal A}$ for
all $i$, where $\mathbb{Z}{\mathcal A}$ denotes the subgroup of
$\mathbb{Z}^n$ spanned by ${\mathcal A}$. 
A major result of Sturmfels \cite[Corollary~8.9]{Stur1} shows that this
triangulation is unimodular if and only if ${\rm in}(P)$ is
square-free. 

\medskip

We are interested in the following:

\begin{conjecture}\label{vila-mfmc-unimodular}\rm\ If $\mathcal C$ is a
uniform clutter that satisfies the max-flow 
min-cut property, then the rational polyhedral cone 
$\mathbb{R}_+\{v_1,\ldots,v_q\}$ has a unimodular regular 
triangulation. 
\end{conjecture} 

\begin{example}\rm Let $u_1,\ldots,u_r$ be the characteristic
vectors of the collection of bases $\mathcal B$ of a transversal matroid
${\mathcal M}$. By \cite[Proposition~2.1 and
Theorem~{4.2}]{blum}, the toric ideal of the subring $K[x^{u_1},\ldots,x^{u_r}]$
has a square-free quadratic Gr\"obner basis. Therefore 
the cone $\mathbb{R}_+\{u_1,\ldots,u_r\}$ or the polytope ${\rm
conv}(u_1,\ldots,u_r)$ has a unimodular regular
triangulation. 
This gives support to 
Conjecture~\ref{vila-mfmc-unimodular} because the clutter 
$\mathcal B$ has the max-flow min-cut property \cite{reesclu}.  
\end{example}

\begin{definition}\rm A matrix $A$ with entries in $\{0,1\}$ is 
called {\it balanced\/} if $A$ has no square submatrix of odd order with exactly
two $1$'s in each row and column.  
\end{definition}

Recall that an integral matrix $A$ 
is $t$-{\it unimodular\/} if all the nonzero $r\times r$
sub-determinants of $A$ have 
absolute value equal to $t$, where 
$r$ is the rank of $A$. If $t=1$ the matrix is called 
{\it unimodular}. If $A$ is $t$-unimodular, then any regular triangulation of
$\mathbb{R}_+\{v_1,\ldots,v_q\}$ is unimodular \cite{Stur1}, see 
\cite[Proposition 5.20]{handbook} for a very short proof of this fact. 

The next result gives an interesting class of uniform clutters, coming
from combinatorial optimization, 
that satisfy Conjecture~\ref{vila-mfmc-unimodular}.  
This result is surprising because not 
all balanced matrices are $t$-unimodular. 

\begin{theorem}\label{balanced->urt} Let $A$ be a balanced matrix with distinct 
column vectors $v_1,\ldots,v_q$. If $|v_i|=d$ for all $i$, then any 
regular triangulation of the cone $\mathbb{R}_+\{v_1,\ldots,v_q\}$ 
is unimodular. 
\end{theorem}

\begin{proof}  Let ${\mathcal A}=\{v_1,\ldots,v_q\}$ and let 
${\mathcal A}_1,\ldots,{\mathcal A}_m$ be the elements of a regular 
triangulation of $\mathbb{R}_+{\mathcal A}$. Then 
$\dim \mathbb{R}_+{\mathcal A}_i=\dim \mathbb{R}_+{\mathcal A}$ and 
${\mathcal A}_i$ is linearly independent for all $i$. Consider 
the clutter $\mathcal{C}$ whose edge ideal is
$I=(x^{v_1},\ldots,x^{v_q})$ and the subclutter ${\mathcal C}_i$ of $\mathcal
C$ whose edges correspond to the vectors in ${\mathcal A}_i$. Let $A_i$ be
the incidence matrix of $\mathcal{C}_i$. Since $A_i$ is
a balanced matrix, using \cite[Corollary~83.1a(iv), p.~1441]{Schr2},
we get that the subclutter ${\mathcal C}_i$ has the max-flow min-cut
property. Hence  
by Theorem~\ref{mfmch->diag} one has $\Delta_r(A_i)=1$, where $r$ is the
rank of $A_i$. Thus the invariant factors of $A_i$ are all equal to $1$
(see \cite[Theorem~3.9]{JacI}). Therefore by the fundamental
structure theorem of finitely generated abelian groups (see
\cite[p.~187]{JacI}) the group $\mathbb{Z}^n/\mathbb{Z}{\mathcal A}_i$ is 
torsion-free for all $i$. Notice that
$\dim\mathbb{R}_+\mathcal{A}={\rm rank}\, \mathbb{Z}\mathcal{A}$ and 
$\dim\mathbb{R}_+\mathcal{A}_i={\rm rank}\, \mathbb{Z}\mathcal{A}_i$
for all $i$. Since $r$ is equal to $\dim \mathbb{R}_+{\mathcal A}$,
it 
follows rapidly that the quotient group
$\mathbb{Z}{\mathcal A}/\mathbb{Z}{\mathcal A}_i$ is torsion-free and has 
rank $0$ for all $i$. Consequently 
$\mathbb{Z}{\mathcal A}=\mathbb{Z}{\mathcal A}_i$ for all $i$, i.e., 
the triangulation is unimodular. \end{proof}

If we do not require that $|v_i|=d$ for all $i$, this result is false
even if $K[F]$ is homogeneous, i.e., even if there is 
$x_0\in\mathbb{R}^n$ such that $\langle v_i,x_0\rangle=1$ 
for all $i$:

\begin{example}\label{balanced-nu-example}\rm Consider the following matrix 
\begin{small}
$$
A=\left(
\begin{array}{ccccccccccccc}
1 &0 &0 &0 &0 &0 &0 &0 &0&  0 &0 &1 &0 \\
0 &1 &0 &0 &0 &0 &0 &0 &0&  0 &0 &1 &0 \\
0 &0 &1 &0 &0 &0 &0 &0 &0&  0 &1 &0 &0 \\
0 &0 &0 &1 &0 &0 &0 &0 &0&  0 &1 &0 &0 \\
0 &0 &0 &0 &1 &0 &0 &0 &0&  1 &0 &0 &0 \\
0 &0 &0 &0 &0 &1 &0 &0 &0&  1 &0 &0 &0 \\
0 &0 &0 &0 &0 &0 &1 &0 &0&  1 &0 &0 &1 \\
0 &0 &0 &0 &0 &0 &0 &1 &0&  0 &1 &0 &1 \\
0 &0 &0 &0 &0 &0 &0 &0 &1&  0 &0 &1 &1 \\
0 &0 &0 &0 &0 &0 &0 &0 &0&  1 &1 &1 &1 
\end{array}
\right)
$$
\end{small}
Let $v_1,\ldots,v_{13}$ be the columns of $A$. It is not hard to see
that the matrix 
$A$ is balanced. Using {\it
Macaulay\/}$2$ 
\cite{mac2} it is seen that 
the regular triangulation $\Delta$ of
$\mathbb{R}_+\{v_1,\ldots,v_{13}\}$ determined by using the 
GRevLex order, on the polynomial ring $K[t_1,\ldots,t_{13}]$, has a
simplex, namely $\{v_1,\ldots,v_{6},v_{10},\ldots,v_{13}\}$, which is not
unimodular.
\end{example}

\bigskip

\noindent
{\bf Acknowledgments.} We thank the referees for their
careful reading of the paper and for the improvements that
they suggested.

\bibliographystyle{plain}

\end{document}